\documentclass[12pt]{amsart}

\usepackage{amssymb,amsthm}

\newtheorem{thm}{Theorem}
\newtheorem{theorem}{Theorem}[section]
\newtheorem{corollary}[theorem]{Corollary}

\newtheorem{lemma}[theorem]{Lemma}

\def\irr#1{{\rm  Irr}(#1)}

\def\ker#1{{\rm ker} (#1)}

\begin{document}

\title[Camina $p$-groups]{Centralizers of Camina $p$-groups of nilpotence class $3$}
\author[Mark L. Lewis]{Mark L. Lewis}

\address{Department of Mathematical Sciences, Kent State University, Kent, OH 44242}
\email{lewis@math.kent.edu}


\begin{abstract}
Let $G$ be a Camina $p$-group of nilpotence class $3$.  We prove that if $G' < C_G (G')$, then $|Z(G)| \le |G':G_3|^{1/2}$.  We also prove that if $G/G_3$ has only one or two abelian subgroups of order $|G:G'|$, then $G' < C_G (G')$.  If $G/G_3$ has $p^a + 1$ abelian subgroups of order $|G:G'|$, then either $G' < C_G (G')$ or $|Z(G)| \le p^{2a}$.
\end{abstract}

\maketitle

\section{Introduction}

Throughout this note, all groups are finite.  A group $G$ is called a {\it Camina group} if the conjugacy class of $g$ is $gG'$ for every element $g \in G \setminus G'$.


Dark and Scoppola have proved in \cite{DaSc} that if $G$ is a Camina group then either $G$ is a Frobenius group  or $G$ is a $p$-group that is a Camina group for some prime $p$.  We say a group is a Camina $p$-group if it is both a Camina group and a $p$-group.
Dark and Scoppola also proved in \cite{DaSc} that if $G$ is a Camina $p$-group, then $G$ has nilpotence class at most $3$.



MacDonald has proved in Theorem 3.1 of \cite{more} that if $G$ is a Camina $2$-group, then $G$ has nilpotence class $2$, so Camina $p$-groups of class $3$ only occur when $p$ is odd.  MacDonald presented examples of Camina $p$-groups of nilpotence class $3$ for every prime $p$ in Theorem 6.3 of \cite{some}.  Other examples have appeared in Theorem 5.1 of \cite{more}, Theorem 2.12 of \cite{MaSc}, and Section 4 of \cite{DaSc}.





In Lemma 2.1 of \cite{more}, MacDonald gave a sketch of a proof that if $G$ is a Camina group of nilpotence class $3$, then $|G_3| \le |G':G_3|$ where $G_3 = [G',G]$.  However, Dark and Scoppolla pointed out in \cite{DaSc} that there is a flaw in MacDonald's argument.  The best result that they could prove is $|G_3| < |G':G_3|^{3/2}$ (see Lemma 5 of \cite{DaSc}).

Our student, Nabil Mlaiki, included in his dissertation \cite{Mlaiki} an argument that also tried to prove MacDonald's claim.  Unfortunately, this argument relied on the assertion that if $G$ is a Camina $p$-group of nilpotence class $3$, then $C_G (G') > G'$.  While it is true that if $G$ is a Camina $p$-group of nilpotence class $3$ and $|G_3| = p$, then $C_G (G') > G'$, this is not true for all Camina $p$-groups of nilpotence class $3$.  In fact, for every odd prime $p$, there exists a Camina $p$-group of nilpotence class $3$ so that $C_G (G') = G'$.  Thus, this argument is fatally flawed.

On the other hand, it is not difficult to see that the argument mentioned in the previous paragraph is valid when one makes the additional assumption that $C_G (G') > G'$.  In particular, if $G$ is a Camina group of nilpotence class $3$ and $C_G (G') > G'$, then $|G_3| \le |G':G_3|$.  In fact, we are now able to prove a stronger result.

\begin{thm} \label{mainfour}
If $G$ is a Camina $p$-group of nilpotence class $3$ and $G' < C_G (G')$, then $|G_3| \le |G':G_3|^{1/2}$.
\end{thm}

Thus, for Camina groups with nilpotence class $3$ and $C_G (G') > G'$, we are able to obtain a bound that is stronger than MacDonald's original bound.  We note that $|G':G_3|$ is a square, so $|G':G_3|^{1/2}$ is necessarily an integer.


We now remove the hypothesis that $G' < C_G (G')$.  Unfortunately, at this time to obtain additional results bounding $|G_3|$, we need to add another hypothesis.  In particular, we add a hypothesis on the number of abelian subgroups of $G/G_3$ that have order $|G:G'|$.

\begin{thm} \label{mainfive}
Let $G$ be a Camina $p$-group of nilpotence class $3$, and let $H = G/G_3$.
\begin{enumerate}
\item If $H$ has one or two abelian subgroups of order $|H:H'|$, then $C_G (G') > G'$.
\item If $H$ has $p^a+1$ abelian subgroups of order $|H:H'|$ for the positive integer $a$, then either $C_G (G') > G'$ or $|G_3| \le p^{2a}$.
\end{enumerate}
\end{thm}

We will see when $G$ is a Camina $p$-group of nilpotence class $3$ that $|G:G'|$ is the maximum size allowed for an abelian subgroup of $G/G_3$.  We will also see that $G/G_3$ is guaranteed to have an abelian subgroup of this size, and in fact, we will use a result of Verardi to see that $G/G_3$ will have either one or two abelian subgroup of this size or it will have $p^a + 1$ abelian subgroups of this size where $a$ is a positive divisor of $n$ and $n$ is defined by $|G:G'| = p^{2n}$.  Notice that when $a < n/2$, Theorem \ref{mainfive} will give a bound that is stronger than MacDonald's bound.

\section{Preliminaries} \label{prelim}

We begin by reviewing some of the known results regarding Camina $p$-groups.  It is not difficult to show that if $G$ is a Camina group and $N$ is a normal subgroup of $G$ that does not contain $G'$, then $G/N$ is a Camina group.  Thus, if $G$ is a Camina $p$-group of nilpotence class $2$, then $G/N$ is a Camina group for every subgroup $N \le Z(G) = G'$ with $|G':N| = p$.  Notice that this implies that $G/N$ will be an extra-special $p$-group.  In fact, Camina $p$-groups of nilpotence class $2$ with $|G'| = p$ are extra-special.  Hence, we can restate the previous fact as saying that if $G$ is a Camina $p$-group with nilpotence class $2$ and $N$ has index $p$ in $Z(G)$, then $G/N$ is extraspecial.  Beisiegel in \cite{beis} has defined a $p$-group $G$ to be {\it semiextraspecial} if $G/N$ is extraspecial for every subgroup $N$ having index $p$ in $Z(G)$.  Verardi proved in Theorem 1.2 of \cite{ver} that every semiextraspecial $p$-group is a Camina $p$-group, and Verardi has extensively studied such groups in \cite{ver}.  This proves that a $p$-group is semiextraspecial if and only if it is a Camina $p$-group of nilpotence class $2$.

In general, a semiextraspecial group $G$ satisfies $|G'|^2 \le |G:G'|$.  (See either Theorem 3.2 of \cite{some} or Theorem 1.5 in \cite{ver}.)  Groups where this maximum is obtained are called {\it ultraspecial}.  Thus, $G$ is an ultraspecial group if $G$ is a semiextraspecial group and $|G:G'| = |G'|^2$. If $G$ is a semiextraspecial group, then the maximum possible order of an abelian subgroup is $|G'| |G:G'|^{1/2}$ (this is Theorem 1.8 of \cite{ver}).  Thus, in an ultraspecial group, the maximum possible order of an abelian subgroup is $|G'|^2 = |G:G'|$.

At the end of Section 3 of \cite{ver}, Verardi constructed an ultraspecial group $K$ that has no abelian subgroups of order $|K:K'|$.  At the beginning of Section 4 of \cite{ver}, Verardi presented an ultraspecial group $K$ with exactly one abelian subgroup of order $|K:K'|$.  In Example 3.9 (b) of \cite{ver}, Verardi shows that there is an infinite family of ultraspecial groups $K$ with exactly two abelian subgroups of order $|K:K'|$.  In Corollary 5.9 of \cite{ver}, Verardi proved that if $K$ is an ultraspecial group with more than two abelian subgroups of order $|K:K'|$, then the number of abelian subgroups of order $|K:K'|$ is $p^h + 1$ where $h$ is a positive integer that divides $n$ and $n$ is defined by $|K:K'| = p^{2n}$.  In Theorem 5.10 of \cite{ver}, it is proved that if $K$ is an ultraspecial group with $|K:K'| = p^{2n}$ and $K$ has $p^n + 1$ abelian subgroups of order $|K:K'|$, then $K$ is isomorphic to a Sylow $p$-subgroup of ${\rm SL}_3 (p^n)$.  We note that a Sylow $p$-subgroup of ${\rm SL}_3 (p^n)$ is the group known as the {\it Heisenberg group} of degree $p^n$.  In summary, an ultraspecial group $K$ has either zero, one, two, or $p^h + 1$ abelian subgroups of order $|K:K'| = p^{2n}$ where $h$ is a divisor of $n$ and $h = n$ if and only if $K$ is the Heisenberg group.

In this paper, we focus on the case where $G$ is a Camina $p$-group of nilpotence class $3$.  This first result was proved in Theorem 5.2 of \cite{some}.

\begin{lemma}\label{basic1}
Let $G$ be a Camina $p$-group of nilpotence class $3$.  Then $G/G_3$ is an ultraspecial group and $|G':G_3|$ is a square.
\end{lemma}

Suppose $a \in G \setminus G'$.  We will define $A(a) = \{ x \in G \mid [a,x] \in G_3 \}$.  Observe that $A(a)/G_3 = C_{G/G_3} (aG_3)$.  Since $G/G_3$ is an ultraspecial group, we have the following corollary.

\begin{corollary}\label{basic5}
Let $G$ be a Camina $p$-group of nilpotence class $3$.  Then $|G:A(a)| = |G':G_3| = |G:G'|^{1/2}$.
\end{corollary}

Under the assumption that $G$ had nilpotence class $3$, we know that $G_3 = [G',G] > 1$ and $[G_3,G] = 1$.  This implies that $G_3 \le Z(G)$.  On the other hand, we write $Z_2 (G)/Z(G) = Z(G/Z(G))$.  It is not difficult to see that $G' \le Z_2 (G)$ and $Z_2 (G) < G$.  In fact, when $G$ is a Camina group, these subgroups will coincide.  The following was proved in Corollary 5.3 of \cite{some}.

\begin{lemma}\label{basic2}
Let $G$ be a Camina $p$-group of nilpotence class $3$.  Then $G_3 = Z(G)$ and $G' = Z_2 (G)$.
\end{lemma}

Let $G$ be a group.  We define $V_1 (G) = V (G)$ to be the product of the subgroups $V (\chi)$ where $\chi$ runs over the nonlinear irreducible characters of $G$ and $V (\chi)$ is the vanishing-off subgroup of $\chi$.  (See page 200 of \cite{text}.)  We say that $V(G)$ is the vanishing-off subgroup of $G$.  It is not difficult to see that $G$ is a Camina group if and only if $G' = V_1 (G)$.  It is useful to defined a a central series associated to $V (G)$ by $V_2 (G) = [V_1 (G), G] \le G'$ and $V_3 (G) = [V_2 (G),G] \le G_3$.  Notice that if $G$ is a Camina group of nilpotence class $3$, then $V_2 (G) = [G',G] = G_3 < G'$ and $V_3 (G) = [G_3,G] = 1 < G_3$.  The following is Theorem 2 (2) and Theorem 3 (3) of \cite{van}.

\begin{theorem}\label{vanishing}
Let $G$ be a solvable group where $V_3 (G) < G_3$ and let $C/G_3 = C_{G/V_3 (G)} (G'/V_3 (G))$, then
\begin{enumerate}
\item $|G:V_1(G)| = |G':V_2 (G)|^2$.
\item Either $|G:C| = |G':V_2 (G)|$ or $C = V_1 (G)$.
\end{enumerate}
\end{theorem}

If $G$ is a Camina group of nilpotence class $3$, then by Theorem 2 (1) of \cite{van}, we know that $G/G' = G/V_1 (G)$, $G'/V_2 (G) = G'/G_3$, and $G_3/V_3 (G) = G_3$ are all elementary abelian $p$-groups for some prime $p$.  We conclude that $G$ is a $p$-group.  The following is an immediate corollary of Theorem \ref{vanishing}.

\begin{corollary} \label{vanone}
Let $G$ be a Camina $p$-group of nilpotence class $3$.  If $C = C_G (G')$, then either $|G:C| = |G':G_3| = |G:G'|^{1/2}$ or $C = G'$.
\end{corollary}

\begin{proof}
Recall that $V_1 (G) = G'$.  We have seen that $V_2 (G) = G_3$ and $V_3 (G) = 1 < G_3$, so Theorem \ref{vanishing} applies.  The conclusion of the corollary are the conclusions of Theorem \ref{vanishing}.
\end{proof}

In our recent paper with Professor Isaacs \cite{fix}, we included a proof of the following lemma.  Since it is known that $Z(G) = G_3$ has exponent $p$, the hypothesis that $Z(G)$ is cyclic is equivalent to $|G_3| = p$.  (This is Proposition 5 of \cite {fix}.)

\begin{lemma} \label{fixone}
Let $G$ be a Camina group with nilpotence class $3$.  If $Z (G)$ is cyclic, then $|C_G (G'):G'| \ge p^2$.
\end{lemma}

Combining Corollary \ref{vanone} with Lemma \ref{fixone}, we obtain the following result which has appeared in the literature in a number of places.  Until now, the only proof of this result has been MacDonald's argument within the proof of Theorem 5.2 of \cite{some}.  Therefore, we have included an independent proof of this fact.

\begin{lemma} \label{fixtwo}
Let $G$ be a Camina group with nilpotence class $3$.  If $|G_3| = p$, then $|C_G (G'):G'| = |G':G_3|$.  In particular, $|G:C_G (G')| = |G:G'|^{1/2}$.
\end{lemma}

\begin{proof}
By Lemma \ref{fixone}, we know that $G' < C_G(G')$.  Using Corollary \ref{vanone}, we obtain the conclusion $|C_G (G'):G'| = |G':G_3|$.
\end{proof}

The following is proved as Proposition 4 of \cite{fix}.  Like the previous lemma, it had previously appeared in the literature but its proof only appeared within the proof of another result.

\begin{lemma} \label{fixthree}
Let $G$ be a Camina group with nilpotence class $3$.  Then $C_G (G')/G_3$ is an elementary abelian $p$-group.
\end{lemma}

\section{${\mathcal A} (G)$}

If $G$ is a Camina $p$-group of nilpotence class $3$ and $a \in G \setminus G'$, recall that $A (a)$ satisfies $A (a)/G_3 = C_{G/G_3} (aG_3)$.  Since $G'/G_3 = Z (G/G_3)$, we see that $G' \le A(a)$ for all $a \in G \setminus G'$.  Also, we know by Corollary \ref{basic5} that $|G:A(a)| = |A(a):G'| = |G':G_3|$.  We begin with the following well-known fact.  Much of the work in this section was motivated by Section 2 of \cite{MaSc}.

\begin{lemma} \label{one}
Let $G$ be a Camina $p$-group of nilpotence class $3$.  If $|G_3|= p$ and $C = C_G (G')$, then $C = A(a)$ for all $a \in C \setminus G'$.
\end{lemma}

\begin{proof}
By Lemma \ref{fixthree}, $C/G_3$ is abelian.  Since $a \in C$, this implies that $C/G_3$ centralizes $aG_3$, and so, $C \le A(a)$.  Applying Corollary \ref{basic5} and Lemma \ref{fixtwo}, we know that $|G:A(a)| = |G:C|$, and we conclude that $C = A(a)$.
\end{proof}

Let $G$ be a Camina $p$-group of nilpotence class $3$.  For $N \le G_3$ such that $|G_3:N| = p$, we define $C(N)/N = C_{G/N} (G'/N)$.  Set ${\mathcal C} (G) = \{ C(N) \mid N \le G_3, |G_3:N| = p \}$.  Observe that if $M < G_3$ and $M \le N$ with $|G_3:N| = p$, then $C (N/M) = C (N)/M$, and so, $\{ C \mid C/M \in {\mathcal C} (G/M) \} \subseteq {\mathcal C} (G)$.

Now, suppose $H$ is a semi-extraspecial $p$-group.  Define ${\mathcal A} (H) = \{ C_H (a) \mid a \in H \setminus H' \}$.  Since $H' = Z (H)$, it follows that $H' \le A$ for every $A \in {\mathcal A} (H)$.

From Lemma \ref{basic1}, we see that if $G$ is a Camina $p$-group of nilpotence class $3$, then $G/G_3$ is an ultraspecial group which implies that $G/G_3$ is semi-extraspecial.   If $a \in G \setminus G'$, then we have defined $A(a)/G_3 = C_{G/G_3} (aG_3)$ and with this in mind we define ${\mathcal A} (G) = \{ A (a) \mid a \in G \setminus G' \}$.  It follows that $A \in {\mathcal A} (G)$ if and only if  $A/G_3 \in {\mathcal A} (G/G_3)$.  Also, if $M < G_3$, then $A/M \in {\mathcal A} (G/M)$ if and only if $A \in {\mathcal A} (G)$.  In light of the previous paragraph, we have $G' \le A$ for every $A \in {\mathcal A} (G)$.


\begin{lemma} \label{two}
Let $G$ be a Camina $p$-group of nilpotence class $3$.
\begin{enumerate}
\item If $A \in {\mathcal A} (G)$ and $A/G_3$ is abelian, then $A = A(a)$ for all $a \in A \setminus G'$.
\item If $C \in {\mathcal C} (G)$, then $C = A(a)$ for all $a \in C \setminus G'$.
\item ${\mathcal C} (G) \subseteq {\mathcal A} (G)$.
\end{enumerate}
\end{lemma}

\begin{proof}
Suppose $A \in {\mathcal A} (G)$ so that $A/G_3$ is abelian.  Consider $a \in A \setminus G'$.  Since $A/G_3$ is abelian, we have $A \le A (a)$.  By Corollary \ref{basic5}, we have $|G:A| = |G:A(a)|$, and hence, $A = A(a)$.  Suppose $C \in {\mathcal C} (G)$.  Thus, there is a subgroup $N$ so that $|G_3:N| = p$ and $C = C(N)$.  By Lemma \ref{one}, we have $C/N \in {\mathcal A} (G/N)$.  This implies that $C/G_3 \in {\mathcal A} (G/G_3)$ and thus, $C \in {\mathcal A} (G)$.  By Lemma \ref{fixthree} applied to $G/N$, we see that $C/G_3 \cong \frac {C/N}{G_3/N}$ is abelian.  This implies that the first conclusion applies to $C$.
\end{proof}

We continue to assume that $G$ is a Camina group of nilpotence class $3$.  Note that if $A (a) = C(N)$, then $A(a)/G_3$ is abelian by Lemma \ref{fixthree}.  With this in mind, we define ${\mathcal A}^* (G) = \{ A \in {\mathcal A} (G) \mid A/G_3 ~is~abelian\}$.  Thus, we have ${\mathcal C} (G) \subseteq {\mathcal A}^* (G) \subseteq {\mathcal A} (G)$.  Since $G_3$ has at least one subgroup of index $p$, it follows that ${\mathcal C} (G)$ is nonempty, and this implies that ${\mathcal A}^* (G)$ is nonempty.  Hence, $G/G_3$ must have at least one centralizer of an element in $G/G_3 \setminus G'/G_3$ that is abelian.  Since $|G:A(a)| = |G:G'|^{1/2}$, we have that $|A(a):G_3| = |G:G'|^{1/2} |G':G_3| = |G:G'|$.  It follows that $G/G_3$ has at least one abelian subgroup of order $|G:G'|$.  We now identify those groups in ${\mathcal A} (G)$ that lie in ${\mathcal C} (G)$.

\begin{lemma}\label{twoa}
Let $G$ be a Camina $p$-group of nilpotence class $3$.  Fix $A \in {\mathcal A} (G)$.
\begin{enumerate}
\item If $N < G_3$ such that $|G_3:N| = p$, then $A = C(N)$ if and only if $[G',A] \le N$.
\item If $M < G_3$, then $A/M \in {\mathcal C} (G/M)$ if and only if $M[G',A] < G_3$.
\item $A \in {\mathcal C} (G)$ if and only if $[G',A] < G_3$.
\end{enumerate}
\end{lemma}

\begin{proof}
Suppose $N < G_3$ such that $|G_3:N| = p$.  If $A = C(N)$, then $A/N$ centralizes $G'/N$, and so, $[G',A] \le N$.  Conversely, suppose that $[G',A] \le N$.  This implies that $A/N$ centralizes $G'/N$, and so, $A \le C(N)$.  Since $|G:A| = |G:C(N)|$, we conclude that $C(N) = A$.

Now, suppose $M < G_3$.  If $A/M \in {\mathcal C} (G/M)$, then there exists $M \le N < G_3$ with $|G_3:N| = p$ so that $A/M = C(N/M)$.  By part (1), this implies that $[G',A] \le N$, and so, $[G',A]M \le N < G_3$.  Conversely, if $[G',A]M < G_3$, then there exists $N$ with $M[G',A] \le N < G_3$ with $|G_3:N| = p$.  This implies that $[G',A] \le N$, so $A = C(N)$ by part (1), and $M \le N$ implies that $C(N/M) = C(N)/M$, so $A/M \in {\mathcal C} (G/M)$.  For conclusion (3), apply conclusion (2) when $M = 1$.
\end{proof}

We now show that ${\mathcal C} (G)$ has only one element if and only if $G' < C_G (G')$.  Furthermore, the one subgroup in ${\mathcal C} (G)$ is $C_G (G')$.

\begin{lemma} \label{four}
Let $G$ be a Camina $p$-group of nilpotence class $3$.  Then $G' < C_G (G')$ if and only if $|{\mathcal C} (G)| = 1$, and if this occurs, then ${\mathcal C} (G) = \{ C_G (G') \}$.
\end{lemma}

\begin{proof}
Suppose that $G' < C_G (G')$ and write $C = C_G (G')$.  Let $N \le G_3$ with $|G_3:N| = p$.  Since $[C,G'] = 1 \le N$, we have $C \le C(N)$.  On the other hand, by Corollary \ref{basic5}, we know that $|C (N):G'| = |G':G_3|$, and by Theorem \ref{vanishing}, we have $|C:G'| = |G':G_3|$.  This implies that $|C(N):G'| = |C:G'|$, and we conclude that $C (N) = C$.

Conversely, suppose that $|{\mathcal C} (G)| = 1$, and write ${\mathcal C} (G) = \{ C \}$.  For every subgroup $N$ such that $|G_3:N|$, we have $C = C(N)$, and hence, $[G',C] \le N$ by Lemma \ref{twoa}.  This implies that $[G'C]$ is contained in the intersection of the subgroup of index $p$.  Since $G_3$ is elementary abelian, this intersection is trivial, so $[C,G'] = 1$ and $G' < C \le C_G (G')$.  By Lemma \ref{vanone}, we have $|G:C_G (G')| =  |G:G'|^{1/2}$.  By Lemma \ref{two}, we know that $C = A(a)$ for any $a \in C \setminus G'$ and by Lemma \ref{basic1} and Corollary \ref{basic5}, we deduce that $|G:C| = |G:G'|^{1/2}$.  Thus, $|G:C_G (G')| = |G:C|$, and we conclude that $C = C_G (G')$.  This proves the lemma.
\end{proof}

Next, we consider intersections between subgroups in ${\mathcal A}^* (G)$ and ${\mathcal A} (G)$.

\begin{lemma} \label{three}
Let $G$ be a Camina $p$-group of nilpotence class $3$.  If $A \in {\mathcal A}^* (G)$ and $B \in {\mathcal A} (G)$ with $A \ne B$, then $A \cap B = G'$ and $G = AB$; in particular, $G/G' = A/G' \times B/G'$.
\end{lemma}

\begin{proof}
Suppose that $c \in A \cap B$, and let $C = A (c)$.  Since $A/G_3$ is abelian and $c \in A$, $A \le C$.  Let $Z/G_3 = Z (B/G_3)$.  We note that $B = A(b)$ for some $b \in G \setminus G'$, and $b \in Z$, so $G_3 < Z$. Since $c \in B$, we have $Z \le C$.  If $b \in A$, then $A \le B$ since $A/G_3$ is abelian, and by Corollary \ref{basic5}, we have $|G:A| = |G:G'|^{1/2} = |G:B|$, and so, $A = B$ which is a contradiction.  Thus, $b$ is not in $A$, and so, $A < AZ \le C$.  This implies that $|G:C| < |G:A| = |G:G'|^{1/2}$.  By Corollary \ref{basic5}, if $c \in G \setminus G'$, then $|G:C| = |G:G'|^{1/2}$, which contradicts $|G:C| < |G:A|$.  This implies that $c \in G'$.  We deduce that $A \cap B \le G'$.  On the other hand, $G' \le A \cap B$, so $A \cap B = G'$.  Now, $|AB:G'| = |A:G'||B:G'| = |G:G'|^{1/2} |G:G'|^{1/2} = |G:G'|$.  We conclude that $G = AB$.
\end{proof}

Let $A$ be a subgroup of $G$ and let $g \in G$.  Then we write $[g,A]$ for the set $\{ [g,a] \mid a \in A \}$.

\begin{lemma}\label{eight}
Let $G$ be a Camina $p$-group of nilpotence class $3$.  If $A = A(a)$ for some element $a \in G \setminus G'$, then $G_3 = [a,A]$.
\end{lemma}

\begin{proof}
Since $C_G (a)$ centralizes $a$ and $G_3 \le C_G (a)$, we see that $C_G (a)/G_3$ will centralize $aG_3$, and so, $C_G (a) \le A(a) = A$.  We know from Corollary \ref{basic5} that $|G:A| = |G':G_3|$.  The usual counting argument yields $|G:C_G (a)| = |{\rm cl}_G (a)| = |aG'| = |G'| = |G:A| |G_3|$.  It follows that $|{\rm cl}_A (a)| = |A:C_G(a)| = |G_3|$.  Since the map $a^x \mapsto [a,x] = a^{-1}a^x$ is a injection from ${\rm cl}_A (a)$ to $[a,A]$, we have $|[a,A]| \ge |G_3|$.  On the hand, since $A/G_3 = C_{G/G_3} (aG_3)$, we know that $[a,A] \le G_3$, and this implies that $G_3 = [a,A]$.
\end{proof}

We now obtain an observation distinguishing the subgroups in ${\mathcal A}^* (G)$ from the remaining subgroups in ${\mathcal A} (G)$.

\begin{lemma}\label{nine}
Let $G$ be a Camina $p$-group of nilpotence class $3$.  If $A \in {\mathcal A} (G)$, then $A \in {\mathcal A}^* (G)$ if and only if $A' = G_3$.
\end{lemma}

\begin{proof}
If $A' = G_3$, then $A/G_3$ is abelian, and $A \in {\mathcal A}^* (G)$.  Conversely, suppose $A \in {\mathcal A}^* (G)$.  We know that if $a \in A \setminus G'$, then $A = A(a)$.  Applying Lemma \ref{eight}, we have $G_3 = [a,A] \le A'$.  On the other hand, since $A \in {\mathcal A}^* (G)$, we have $A/G_3$ is abelian, so $A' \le G_3$.  We conclude that $A' = G_3$.
\end{proof}

Note that this next theorem includes Theorem \ref{mainfour} from the Introduction.  Following \cite{GCG}, we say that a group $G$ is a VZ-group if every nonlinear irreducible character of $G$ vanishes off of the center of $G$.

\begin{theorem} \label{ten}
Let $G$ be a Camina $p$-group of nilpotence class $3$.  Suppose that $C = C_G (G') > G'$. Then:
\begin{enumerate}
\item $C' = G_3$ and $Z(C) = G'$.
\item $C$ is a VZ-group.
\item $|G_3| \le |G':G_3|^{1/2}$.
\end{enumerate}
\end{theorem}

\begin{proof}
We know via Lemma \ref{one} that $C = A(a)$ for every $a \in C \setminus G'$.  By Lemma \ref{fixthree}, we know that $C/G_3$ is abelian; so by Lemma \ref{nine} we have that $C' = G_3$.  Since $C = C_G (G')$, it follows that $G' \le Z (C)$.  On the other hand, if $a \in C \setminus G'$, then Lemma \ref{eight} yields $[a,C] = G_3 \ne 1$, so $a \not\in Z (C)$.  We deduce that $Z (C) \le G'$, and so, $G' = Z(C)$.  We now have that ${\rm cl}_C (a) = aC'$ for every element $a \in C \setminus Z(C)$, so applying Lemma 2.2 of \cite{GCG}, we have that $C$ is a VZ-group.  By Theorem \ref{vanishing}, we have that $|C:Z(C)| = |C:G'| = |G':G_3|$.  Applying Lemma 2.4 of \cite{GCG}, we obtain $|G_3| = |C'| \le |C:Z(C)|^{n/2} = |G':G_3|^{1/2}$.
\end{proof}

We now apply this result to obtain the following corollary when $|G_3| > |G':G_3|^{1/2}$.

\begin{corollary} \label{ten1}
Suppose $G$ is a Camina $p$-group with nilpotence class $3$.  If $|G_3| > |G':G_3|^{1/2}$, then $C_G (G') = G'$.
\end{corollary}

\begin{proof}
By the contrapositive of Theorem \ref{ten} (3), we see that $|G_3| > |G':G_3|^{1/2}$ implies that $C_G (G') \le G'$.  By Corollary \ref{vanone}, $C_G (G') = G'$.
\end{proof}

We now use the examples in Theorem 2.10 of \cite{MaSc} and Section 4 of \cite{DaSc} to show that there exist Camina $p$-groups of nilpotence class $3$ with $C_G (G') = G'$.  

\begin{lemma}
For every odd prime $p$, there exists a Camina $p$-group of nilpotence class $3$ with $C_G (G') = G'$.
\end{lemma}

\begin{proof}
By Corollary \ref{ten1}, we just need a Camina $p$-group of nilpotence class $3$ with $|G_3| > p^{n/2}$ where $|G:G'| = p^{2n}$.  Such groups are constructed in Theorem 2.10 of \cite{MaSc} and Section 4 of \cite{DaSc}.
\end{proof}

\section {${\mathcal C} (G)$}

We begin this section with some general results.  Let $G$ be a group.  A set of subgroups $\{ A_1, \dots, A_k \}$ is a {\it partition} of $G$ if $G = \bigcup_{i=1}^k A_i$ and $A_i \cap A_j = 1$ whenever $i \ne j$.  We say the partition is nontrivial if $k > 1$ and $A_i > 1$ for all $i = 1, \dots, k$.

\begin{lemma} \label{countone}
Let $A$ be an elementary abelian $p$-group of order $p^n$.  Let $\{ A_1, \dots, A_k \}$ be a nontrivial partition of $A$.  If $n = 2l$ for some integer $l$, then $k \ge p^l + 1$ with equality occurring if and only if all $|A_i| = p^l - 1$ and if $n = 2l + 1$ for some integer $l$, then $k > p^l + 1$.
\end{lemma}

\begin{proof}
Since the $A_i$'s form a partition, we have that $|G| - 1 = \sum_{i=1}^k (|A_i| - 1)$.  Notice that for each $i$, the order $|A_i|$ is a power of $p$ between $p$ and $p^{n-1}$.  For $j = 1, \dots, n-1$, we set $a_j = |\{ i \mid |A_i| = p^j \}|$.  I.e., $a_j$ is the number of $A_i$'s that have order $p^j$.  It follows that $|G| - 1 = p^n - 1 = \sum_{j=1}^{n-1} a_j (p^j - 1)$ and $k = \sum_{j=1}^{n-1} a_j$.  Without loss of generality, we may assume that $|A_1| \ge |A_i|$ for all $i = 2, \dots, k$.

Suppose first that $|A_1| > p^l$.  Thus, we can write $|A_1| = p^b$ where $b \ge l +1$.  We know that $A_1 \cap A_i = 1$ for $i = 2, \dots, k$.  This implies that $|A_i| = |A_1A_i|/|A_1| \le |P|/|A_1| = p^n/p^b = p^{n-b}$.  This implies that $a_b = 1$, $a_j = 0$ when $j > n-b$ and $j \ne b$, and $\sum_{j=1}^{n-b} a_j = k-1$.  We now have
$$
|G| - 1 = p^n - 1 = (p^b - 1) + \sum_{j=1}^{n-b} a_j (p^j - 1) \le (p^b - 1) + (k-1) (p^{n-b} - 1).
$$
Rearranging this equation, we have $(p^n - 1) - (p^b - 1) \le (p^{n-b} - 1)$.  Observe that $(p^n - 1) - (p^b - 1) = p^b (p^{n-b} - 1)$.  This yields $p^b (p^{n-b} - 1) \le (k-1) (p^{n-b} - 1)$.  Dividing by $p^{n-b} - 1$, we obtain $p^b \le k- 1$.  We conclude that $p^l + 1 < p^b + 1 \le k$, as desired.  Notice that we cannot have equality in this case.

Thus, we may now assume that $|A_1| \le p^l$, and so, $|A_i| \le p^l$ for $i = 1, \dots, k$.  This implies that $a_j = 0$ when $j > l$.  We now have
$$
|G| - 1 = p^n - 1 = \sum_{j=1}^l a_j (p^j - 1) \le k (p^l - 1).
$$
Dividing by $p^l - 1$, we see that $k \ge (p^n - 1)/(p^l - 1)$.  If $n = 2l$, then $(p^n-1)/(p^l-1) = (p^{2l} - 1)/(p^l - 1) = p^l + 1$.  This yields $k \ge p^l + 1$ when $n = 2l$.  We obtain equality if and only if $|A_i| = p^l - 1$ for all $i = 1,\dots, k$.  If $n = 2l + 1$, then $k \ge (p^{2l+1}-1)/(p^l - 1) > (p^{2l} - 1)/(p^l-1) = p^l + 1$.  This completes the proof.
\end{proof}

We now obtain a dual result for Lemma \ref{countone}.

\begin{lemma} \label{counttwo}
Let $M$ be an elementary abelian $p$-group of order $p^n$.  Let $M_1, \dots, M_k$ be proper, nontrivial subgroups of $M$ that satisfy the following two conditions: (1) $M = M_i M_j$ whenever $1 \le i \ne j \le k$ and (2) if $N < M$ with $|M:N| = p$, then there exists some $i$ with $1 \le i \le k$ so that $M_i \le N$.  If $n = 2l$, then $k \ge p^l + 1$ and if $n = 2l+1$, then $k > p^l + 1$.
\end{lemma}

\begin{proof}
By Lemma \ref{countone}, it suffices to show that the following sets of characters $\{ \irr {M/M_1}, \dots, \irr {M/M_k} \}$ is a nontrivial partition of $\irr M$.  Suppose $i \ne j$ and $\lambda \in \irr {M/M_i} \cap \irr {M/M_j}$.  This implies that $M_i \le \ker {\lambda}$ and $M_j \le \ker {\lambda}$.  We deduce that $M_i M_j \le \ker {\lambda}$.  Since $i \ne j$, we have $G = M_i M_j$, and so, $M \le \ker {\lambda}$.  Thus, $\lambda = 1_M$, and we see that $\irr {M/M_i} \cap \irr {M/M_j} = \{ 1_M \}$.  For any $\lambda \in \irr M$, we know that $|M:\ker {\lambda}| = p$.  Thus, there is an integer $i$ with $1 \le i \le k$ so that $M_i \le \ker {\lambda}$, and so, $\lambda \in \irr {M/M_i}$.  We conclude that $\irr M = \bigcup_{i=1}^k \irr {M/M_i}$.  Since $M_i < M$ for each $i$, we see that $\irr {M/M_i} > \{ 1_M \}$ for all $i$, and since $M_i > 1$, we have $\irr {M/M_i} < \irr M$ for all $i$. Therefore, we have the desired nontrivial partition, and the lemma is proved.
\end{proof}

We now consider the case where $G$ is a Camina group of nilpotence class $3$ and ${\mathcal C} (G)$ contains at least two elements.

\begin{lemma}\label{twob}
Let $G$ be a Camina $p$-group of nilpotence class $3$.  If $C_1, C_2 \in {\mathcal C} (G)$ satisfy $C_1 \ne C_2$, then $G_3 = [C_1,G'] [C_2,G']$
\end{lemma}

\begin{proof}
Observe since ${\mathcal C} (G) \subseteq {\mathcal A} (G)$, we have $C_1,C_2 \in {\mathcal A} (G)$.  Suppose $[C_1,G'][C_2,G'] < G_3$.  Then you can find a subgroup $N$ so that $[C_1,G'][C_2,G'] \le N$ and $|G_3:N| = p$.  By Lemma \ref{twoa}, we have $C_1 = C(N) = C_2$ which is a contradiction.  Thus, $G_3 = [C_1,G'][C_2,G']$.
\end{proof}

With this in hand, we obtain a count of ${\mathcal C} (G)$ in terms of $|G_3|$.

\begin{lemma} \label{countthree}
Let $G$ be a Camina $p$-group of nilpotence class $3$.  Suppose that $C_G (G') = G'$ and $|G_3| = p^m$.  If $m = 2l$, then $|{\mathcal C} (G)| \ge p^l + 1$ and if $m = 2l + 1$, then $|{\mathcal C} (G)| > p^l + 1$.
\end{lemma}

\begin{proof}
Let ${\mathcal C} (G) = \{ C_1, \dots, C_k \}$ where the $C_i$'s are distinct.  By Lemma \ref{four}, since $C_G (G') = G'$, we must have $k > 1$.  Since $C_G (G') = G' < C_i$, we have $1 < [C_i,G']$.  On the other hand, applying Lemma \ref{twoa} (3), we see that $[C_i,G'] < G_3$.  We use Lemma \ref{twob} to see that $[C_i,G'][C_j,G'] = G_3$ whenever $1 \le i \ne j \le k$.  Finally, if $N < G_3$ with $|G_3:N| = p$, then $C(N) \in {\mathcal C} (G)$, so $C (N) = C_i$ for some $i$ and from Lemma \ref{twoa} (1), $[C_i,G'] \le N$.  Thus, the subgroups $[C_1,G'], \dots, [C_k,G']$ satisfy the hypotheses of Lemma \ref{counttwo} in $G_3$.  Now, the conclusion of this lemma is exactly the conclusion of Lemma \ref{counttwo}.
\end{proof}

\begin{lemma} \label{eleven}
Let $G$ be a Camina $p$-group of nilpotence class $3$.  If $|{\mathcal C} (G)| \le p$, then $C_{G} (G') > G'$ and in particular, $|{\mathcal C} (G)| = 1$.
\end{lemma}

\begin{proof}
Suppose that $C_G (G') = G'$.  Let $|G_3| = p^m$ where $m \ge 1$ is an integer.  By Lemma \ref{fixtwo}, we know that $C_G (G') = G'$ implies that $m \ge 2$.  Thus, either $m = 2l$ or $m = 2l + 1$ where $l \ge 1$, and $|{\mathcal C} (G)| \ge p^l + 1 > p$ by Lemma \ref{countthree}, a contradiction.  The second conclusion is Lemma \ref{four}.
\end{proof}

We noted before Lemma \ref{twoa}, ${\mathcal A}^* (G)$ is nonempty.  Also, by Lemma \ref{basic1}, $G/G_3$ is an ultraspecial group.  It is not difficult to see that ${\mathcal A}^* (G/G_3)$ will consist of the abelian subgroups of $G/G_3$ having order $|G:G'|$, and recall that ${\mathcal A}^* (G/G_3) = \{ A/G_3 \mid A \in {\mathcal A}^* (G) \}$.  Using the results of Verardi mentioned in the third paragraph of Section \ref{prelim}, we see that $|{\mathcal A}^* (G)|$ is either $1$, $2$ or it has the value $1 + p^h$ for some integer $h$ with $h$ dividing $n$.  %
\begin{theorem} \label{countfour}
Let $G$ be a Camina $p$-group of nilpotence class $3$.  Suppose that $C_G (G') = G'$.  If $|{\mathcal C} (G)| \le p^a + 1$ for some integer $a \ge 1$, then $|G_3| \le p^{2a}$.
\end{theorem}

\begin{proof}
Let $|G_3| = p^m$.  In light of Lemma \ref{fixtwo}, we know that $m \ge 2$.  Suppose first that $m = 2l$.  By Lemma \ref{countthree}, we have $p^l + 1 \le |{\mathcal C} (G)|$.  This implies that $p^l + 1 \le p^a + 1$ and so $l \le a$.  This yields $m \le 2a$ as desired.  Now suppose that $m = 2l + 1$.  By Lemma \ref{countthree}, we have $p^l + 1 < |{\mathcal C} (G)|$.  This implies that $p^l + 1 < p^a + 1$, and so, $l < a$.  We obtain $l \le a-1$, and so, $m = 2l + 1 \le 2(a-1) + 1 = 2a - 1 < 2a$ as desired.
\end{proof}

We now prove Theorem \ref{mainfive}.

\begin{proof}[Proof of Theorem \ref{mainfive}]
We know that ${\mathcal C} (G) \subseteq {\mathcal A}^* (G)$.  Thus, if $|{\mathcal A}^* (G)| = 1$ or $2$, then $|{\mathcal C} (G)| \le 2 < p$, and so Lemma \ref{eleven} will imply that $G' < C_G (G')$.  This yields conclusion (1).  Now, suppose that $|{\mathcal A}^* (G)| = p^a + 1$, and thus, $|{\mathcal C} (G)| \le p^a + 1$.  If $G' < C_G (G')$, then conclusion (2) holds and we are done.  Thus, we may assume that $G' = C_G (G')$. This implies that we are in the situation of Theorem \ref{countfour}, and conclusion (2) follows from that theorem.
\end{proof}

\section {$G' = C_G (G')$}

We conclude by gathering some facts about $G$ when $G$ is a Camina $p$-group of nilpotence class $3$ that satisfies $G' = C_G (G')$.  In particular, we will draw some conclusions when $|G_3| \ge |G':G_3|$.  We begin by showing that elements in ${\mathcal A} (G)$ do not contain the centralizers of elements in $G' \setminus G_3$.

\begin{lemma}\label{twelve}
Let $G$ be a Camina $p$-group of nilpotence class $3$.  If there exists $g \in G' \setminus G_3$ and $A \in {\mathcal A} (G)$ such that $A \le C_G (g)$, then $A = C_G (G')$.
\end{lemma}

\begin{proof}
Suppose $N \le G_3$ satisfies $|G_3:N| = p$.  Let $C/N = C_{G/N} (gN)$.  Observe that $A \le C_G (g) \le C$.  On the other hand, since $gN \in G'/N$, we must have that $C/N \ge C_{G/N} (G'/N) = C(N)/N$; so $C(N) \le C$.  We know from Lemma \ref{three} that if $A \ne C(N)$, then $G = A C(N)$.  Since $A C(N) \le C$, this would imply that $C = G$ which is a contradiction since $Z (G/N) = G_3/N$.  It follows that $A = C(N)$ for all possible $N$ with $N \le G_3$ and $|G_3:N| = p$.  In particular, ${\mathcal C} (G) = \{ A \}$.  By Lemma \ref{four}, we conclude that $A = C_G (G')$.
\end{proof}

\begin{corollary}\label{thirteen}
Let $G$ be a Camina $p$-group of nilpotence class $3$.  If $A \in {\mathcal A} (G)$ satisfies $A \ne C_G (G')$, then $A \not\le C_G (g)$ for all $g \in G' \setminus G_3$.  In particular, if $C_G (G') = G'$, then $A \not \le C_G (g)$ for all $A \in {\mathcal A} (G)$ and for all $g \in G' \setminus G_3$.
\end{corollary}

\begin{proof}
This is the contrapositive of Lemma \ref{twelve}.
\end{proof}

We now characterize the subgroups in ${\mathcal A}^* (G) \setminus {\mathcal C} (G)$.  In this first lemma, we handle the case where $|G_3| = p$.

\begin{lemma}\label{fourteen}
Let $G$ be a Camina $p$-group of nilpotence class $3$, and suppose that $|G_3| = p$.  If $A \in {\mathcal A}^* (G)$ and $A \ne C_G (G')$, then $A$ is an extraspecial group.
\end{lemma}

\begin{proof}
We know from Lemma \ref{nine} that $A' = G_3$.  We know that if $a \in A \setminus G'$, then $A = A (a)$.  Applying Lemma \ref{eight}, we have $[a,A] = G_3$, and so, $Z(A) \le G'$.  In light of Corollary \ref{thirteen}, we have that if $g \in G' \setminus G_3$, then $A \not\le C_G (g)$.  This implies that $g \not\in Z(A)$, and we deduce that $Z(A) \le G_3$.  Since $G_3 = Z(G)$, this implies $G_3 = Z(A)$.  Also, $A/G_3$ is a subgroup of $G/G_3$ which has exponent $p$, and we have $A' = Z(A) = \Phi (A) = G_3$.  Because $|G_3| = p$, we conclude that $A$ is extra-special.
\end{proof}

We now have the general case.

\begin{corollary}\label{fifteen}
Let $G$ be a Camina $p$-group of nilpotence class $3$.  If $A \in {\mathcal A}^* (G) \setminus {\mathcal C} (G)$, then $A$ is a semi-extraspecial group.
\end{corollary}

\begin{proof}
Let $N \le G_3$ satisfy $|G_3:N| = p$.  Applying Lemma \ref{fourteen} to $G/N$, we see that $A/N$ is an extra-special group, and we conclude that $A$ is semi-extraspecial.
\end{proof}

We also know from \cite{mann} that $G/G_3$ has exponent $p$.  Combining this fact with Theorem 6.1 of \cite{ver}, we deduce that ${\mathcal A}^* (G) = {\mathcal A} (G)$ if and only if $G/G_3$ is isomorphic to the Heisenberg group for $p^n$.

\begin{corollary}\label{sixteen}
Let $G$ be a Camina $p$-group of nilpotence class $3$.  If ${\mathcal C} (G) < {\mathcal A} (G)$, then $|G_3| \le |G:G'|^{1/2}$.
\end{corollary}

\begin{proof}
Recall that ${\mathcal C} (G) \subseteq {\mathcal A}^* (G)$, and $|{\mathcal A}^* (G)|$ is either $1$, $2$, or $1 + p^a$ where $a$ divides $n$ and $n$ satisfies $|G:G'| = p^{2n}$.  Thus, if ${\mathcal A}^* (G) < {\mathcal A} (G)$, then the possible values of $a$ are at most $n/2$.  Applying Theorems \ref{mainfour} and \ref{mainfive}, we see that either $|G_3| \le |G:G'|^{1/4}$ or $|G_3| \le p^{2a} \le p^n = |G:G'|^{1/2}$. Thus, we may assume that ${\mathcal A}^* (G) = {\mathcal A} (G)$, and there exists $A \in {\mathcal A}^* (G) \setminus {\mathcal C} (G)$.  By Lemma \ref{fifteen}, we know that $A$ is semi-extraspecial, and from \cite{ver}, it follows that $|G_3| = |A'| \le |G:G'|^{1/2}$.
\end{proof}

Finally, we make an observation when $G$ is a Camina $p$-group of nilpotence class $3$ and $|G_3| > |G:G'|^{1/2}$.

\begin{corollary}\label{seventeen}
Let $G$ be a Camina $p$-group of nilpotence class $3$.  If $|G:G'| = p^{2n}$ and $|G_3| > p^n$, then ${\mathcal C} (G) = {\mathcal A}^* (G) = {\mathcal A} (G)$ and $|{\mathcal C} (G)| = p^n + 1$.
\end{corollary}

\begin{proof}
We know by Corollary \ref{sixteen} that ${\mathcal C} (G) = {\mathcal A} (G)$, and since ${\mathcal C} (G) \subseteq {\mathcal A}^* (G) \subseteq {\mathcal A} (G)$, we have the first conclusion.  In Theorem 1.3 (vi) of \cite{MaSc}, it is proved that if $H$ is a Camina $p$-group of nilpotence class $3$ and $|H_3| = |H:H'|^{1/2}$, then $H/H_3$ is isomorphic to the Heisenberg group.  We can find a subgroup $N$ in $G_3$ so that $|G_3:N| = p^n$.  Then $H \cong G/N$ is a Camina $p$-group of nilpotence class $3$ and $|H_3| = |H:H'|^{1/2}$.   It follows that $G/G_3 \cong H/H_3$ is isomorphic to the Heisenberg group.  Since $G/G_3$ is isomorphic to the Heisenberg group, we know that $|{\mathcal A} (G)| = p^n + 1$, and this yields the second conclusion.
\end{proof}

\section{Final remarks}

As we have mentioned before, Camina groups having nilpotence class $3$ have been constructed in Theorem 6.3 of \cite{some}, Theorem 5.1 of \cite{more}, Theorem 2.12 of \cite{MaSc}, and Section 4 of \cite{DaSc}.  However, in all of these cases if $G$ is the nilpotence class $3$ Camina group constructed, then $G/Z(G)$ is a Heisenberg group.  It makes sense to ask whether there exist Camina groups $G$ of nilpotence class $3$ where $G/G_3$ is not a Heisenberg group, since if no such groups exist then Theorem \ref{mainfive} would be pointless.  Thus, in the appendix, we present code from the computer algebra system Magma \cite{magma} for two Camina $3$-groups $G$ of nilpotence class $3$ one of which has order $3^{13}$ and the other has order $3^{14}$.  In both cases $|{\mathcal A}^* (G)| = 1$.  We note that $G/G_3$ is not isomorphic in the two examples since the first example has $|{\mathcal A} (G)| = 3241$ and the second example has $|{\mathcal A} (G)| = 2998$.

This raises the question of which groups can occur as $H = G/G_3$ when $G$ is a Camina group of nilpotence class $3$.  By Lemma \ref{basic1}, we know that $H$ must be ultraspecial.  In light of the comments following Lemma \ref{two}, we must have $|{\mathcal A}^* (H)| \ge 1$.  Finally, as mentioned before, Mann proved in Theorem 1 (b) of \cite{mann} that $H$ has exponent $p$.  Thus, it seems reasonable to ask what other conditions (if any) are necessary for $H$.  At this time, we do not have a satisfactory answer to this question.

\bigskip
\bigskip

\center{Appendix}

\bigskip

\begin{verbatim}
p1 := PolycyclicGroup<x1,x2,x3,x4,x5,x6,x7,x8,x9,x10,x11,
x12,x13| x1^3,x2^3,x3^3,x4^3,x5^3,x6^3,x7^3,x8^3,x9^3,
x10^3,x11^3,x12^3,x13^3, x2^x1=x2*x13^2, x3^x1=x3*x13, 
x3^x2=x3*x13, x4^x1=x4*x13,x5^x1=x5*x9*x13^2, 
x5^x2=x5*x10, x5^x3=x5*x11*x12^2*x13, 
x5^x4=x5*x11^2*x13^2, x6^x1=x6*x10*x13^2, x6^x2=x6*x9*x10, 
x6^x3=x6*x11^2, x6^x4=x6*x12^2, x6^x5=x6*x13, 
x7^x1=x7*x11*x13, x7^x2=x7*x12*x13^2, 
x7^x3=x7*x9*x11^2*x12^2*x13,x7^x4=x7*x10*x11^2*x12*x13^2, 
x7^x5=x7*x13, x7^x6=x7*x9*x10*x13, x8^x1=x8*x12*x13,
x8^x2=x8*x11*x12*x13, x8^x3=x8*x10*x11^2*x12, 
x8^x4=x8*x9*x10*x11, x8^x5=x8*x13, x8^x6=x8*x9*x11, 
x8^x7=x8*x10*x12, x9^x5=x9*x13, x9^x6=x9*x13^2, 
x9^x7=x9*x13^2,x9^x8=x9*x13, x10^x5=x10*x13^2, 
x10^x7=x10*x13, x11^x5=x11*x13^2, x11^x6=x11*x13, 
x11^x7=x11*x13^2, x11^x8=x11*x13, x12^x5=x12*x13, 
x12^x7=x12*x13>;

p2 := PolycyclicGroup<x1,x2,x3,x4,x5,x6,x7,x8,x9,x10,x11,
x12,x13,x14| x1^3,x2^3,x3^3,x4^3,x5^3,x6^3,x7^3,x8^3,
x9^3,x10^3,x11^3,x12^3,x13^3,x14^3,x2^x1=x2*x14^2, 
x3^x1=x3*x14, x3^x2=x3*x13, x4^x1=x4*x13^2, 
x4^x2=x4*x13*x14^2, x4^x3=x4*x13^2*x14^2, x5^x1=x5*x9*x13, 
x5^x2=x5*x10*x13^2, x5^x3=x5*x11*x12^2*x13^2*x14^2, 
x5^x4=x5*x11^2*x13*x14,x6^x1=x6*x10, 
x6^x2=x6*x9*x10*x13*x14^2, x6^x3=x6*x11^2*x13^2*x14, 
x6^x4=x6*x12^2, x6^x5=x6*x14, x7^x1=x7*x11*x13^2, 
x7^x2=x7*x12*x13^2*x14^2, 
x7^x3=x7*x9*x11^2*x12^2*x13*x14^2, 
x7^x4=x7*x10*x11^2*x12*x13*x14, x7^x5=x7*x13^2*x14^2, 
x7^x6=x7*x13, x8^x1=x8*x12, x8^x2=x8*x11*x12, 
x8^x3=x8*x10*x11^2*x12, x8^x4=x8*x9*x10*x11, 
x8^x5=x8*x14, x8^x6=x8*x13^2*x14^2, x8^x7=x8*x9, 
x9^x7=x9*x13^2, x9^x8=x9*x13*x14^2, x10^x7=x10*x13*x14^2, 
x10^x8=x10*x14^2, x11^x5=x11*x13^2, x11^x6=x11*x13*x14^2, 
x11^x7=x11*x13*x14, x11^x8=x11*x13, x12^x5=x12*x13*x14^2, 
x12^x6=x12*x14^2, x12^x7=x12*x13, x12^x8=x12*x13^2*x14>;
\end{verbatim}

\bigskip
\bigskip

\end{document}